\documentclass{svproc}
\usepackage{amsmath}
\usepackage{amssymb}
\usepackage{comment}
\usepackage{hyperref}
\newcommand{\bburl}[1]{\textcolor{blue}{\url{#1}}}
\newcommand{\seqnum}[1]{\href{https://oeis.org/#1}{\rm \underline{#1}}}

\usepackage{url}

\begin{document}
\mainmatter              
\title{On a Pair of Diophantine Equations}
\titlerunning{On a Pair of Diophantine Equations}  
%
\author{Sujith Uthsara Kalansuriya Arachchi\inst{1} \and H\`ung Vi\d{\^e}t Chu\inst{2}
Jiasen Liu\inst{3} \and Qitong Luan\inst{4} \and Rukshan Marasinghe\inst{5} \and Steven J. Miller\inst{6}}
\authorrunning{Arachchi et al.} 
%
\tocauthor{Sujith Uthsara Kalansuriya Arachchi, H\`ung Vi\d{\^e}t Chu, Jiasen Liu, 
Qitong Luan, Rukshan Marasinghe, Steven J. Miller}
\institute{Department of Mathematics, Faculty of Science, University of Colombo, Cumarathunga Munidasa Mawatha, Colombo 00300, Sri Lanka,\\
\email{sujithuthsarak@gmail.com}
\and
Department of Mathematics, Washington and Lee University, Lexington, VA 24450, USA,\\
\email{hchu@wlu.edu}
\and
Department of Mathematics, University of Southern California, Los Angeles, CA 90007, USA,\\
\email{jiasen.jason.liu@gmail.com}
\and
Department of Mathematics, University of California, Los Angeles, CA 90095, USA,\\
\email{qluan21@g.ucla.edu}
\and
Department of Mathematics, Faculty of Science, University of Colombo, Cumarathunga Munidasa Mawatha, Colombo 00300, Sri Lanka,\\
\email{rukshanmarasinghe@gmail.com}
\and
Department of Mathematics and Statistics, Williams College, Williamstown, MA 01267, USA,\\
\email{sjm1@williams.edu}; \email{Steven.Miller.MC.96@aya.yale.edu}
}

\maketitle              

\begin{abstract}
For relatively prime natural numbers $a$ and $b$, we study the two equations $ax+by = (a-1)(b-1)/2$ and $ax+by+1= (a-1)(b-1)/2$, which arise from the study of cyclotomic polynomials. Previous work showed that exactly one equation has a nonnegative solution, and the solution is unique. Our first result gives criteria to determine which equation is used for a given pair $(a,b)$. We then use the criteria to study the sequence of equations used by the pair $(a_n/\gcd{(a_n, a_{n+1})}, a_{n+1}/\gcd{(a_n, a_{n+1})})$ from several special sequences $(a_n)_{n\geq 1}$. Finally, fixing $k \in \mathbb{N}$, we investigate the periodicity of the sequence of equations used by the pair $(k/\gcd{(k, n)}, n/\gcd{(k, n)})$ as $n$ increases.  
\keywords{Diophantine equation, sequence, periodic}
\end{abstract}

\section{Introduction}

Given $a, b\in \mathbb{N}$, the following two Diophatine equations arise in the study of cyclotomic polynomials $\Phi_{pq}(x)$ \footnote{The equations are used in calculating the midterm coefficient of $\Phi_{pq}(x)$ (see \cite[pp.\ 770]{B}).} for primes $p < q$ \cite{B}: 
\begin{align}
\label{e1} ax + by &\ =\  \frac{(a-1)(b-1)}{2},\\
\label{e2} ax + by + 1 &\ =\ \frac{(a-1)(b-1)}{2}.
\end{align}
Beiter \cite{B} showed that if $a, b$ are primes, then exactly one of the equations has a nonnegative integral solution $(x,y)$. Chu \cite[Theorem 1.1]{C} extended the result to any pair of relatively prime numbers $(a, b)$; furthermore, the solution is unique. 
For a fixed pair of relatively prime numbers $(a,b)\in \mathbb{N}^2$, we say that $(a,b)$ uses Equation \eqref{e1} if \eqref{e1} has a nonnegative solution; otherwise, we say that $(a,b)$ uses Equation \eqref{e2}. As a corollary of \cite[Theorem 1.1]{C}, Chu considered which equation is used by the pairs $(F_n, F_{n+1})$ and $(F_n, F_{n+2})$, where $(F_n)_{\ge 1}$ is the Fibonacci sequence\footnote{Note that $\gcd(F_n, F_{n+1}) = \gcd(F_n, F_{n+2}) = 1$.}, and established new identities involving Fibonacci numbers. In particular, \cite[Theorems 1.4 and 1.6]{C} stated that the pair $(F_n, F_{n+1})$ uses \eqref{e1} and \eqref{e2} alternatively in groups of three; the same conclusion holds for $(F_n, F_{n+2})$. Following the work, Chen et al.\ studied which equation is used by $(F_n^2, F_{n+1}^2)$ and $(F_n^3, F_{n+1}^3)$ and discovered the following identities: for $n\ge 2$,
\begin{align*}
    1 + \frac{F_n^2-3}{2}F_n^2 + \frac{F_n^2 - F^2_{n-1}-1}{2}F_{n+1}^2 &\ =\ \frac{(F_n^2-1)(F_{n+1}^2-1)}{2}\,, \mbox{ if }n\mbox{ is odd},\\
    1 + \frac{F_n^2+1}{2}F_n^2 + \frac{F_n^2 - F^2_{n-1}-1}{2}F^2_{n+1}&\ =\ \frac{(F_n^2-1)(F^2_{n+1}-1)}{2}\,, \mbox{ if }n\mbox{ is even},\\
    \left(F_n^2 - \frac{F^2_{n-1}+1}{2}\right)F_n^2 + \frac{F^2_{n-1}-1}{2}F^2_{n+1}&\ =\ \frac{(F_n^2-1)(F^2_{n+1}-1)}{2}\,,
\end{align*}
and
\begin{align*}
    \left(\sum_{k=1}^{2n-1}(-1)^{k-1} F_k^3\right)F^3_{2n-1} + \left(\sum_{k=2}^{2n-2}F_k^3\right)F^3_{2n}&\ =\ \frac{(F^3_{2n-1}-1)(F^3_{2n}-1)}{2}\,,\\
    1 + \left(\sum_{k=1}^{2n}(-1)^k F_k^3-1\right)F^3_{2n} + \left(\sum_{k=2}^{2n-1}F_k^3\right)F^3_{2n+1}&\ =\ \frac{(F^3_{2n}-1)(F^3_{2n+1}-1)}{2}\,.
\end{align*}

Continuing these work, the present paper studies the sequence of equations used by consecutive terms of some special sequences $(a_n)_{n\ge 1}$. Moreover, we avoid requiring consecutive terms of $(a_n)_{n\ge 1}$ to be relatively prime (as in the case of $(F_n)_{n\ge 1}$) by considering $(a_n/d, a_{n+1}/d)$ for $d = \gcd(a_n, a_{n+1})$ instead. This was recently done by Davala \cite{D} for $(B_n, B_{n+2})$, where $(B_n)_{\ge 1}$ is the balancing sequence\footnote{Balancing numbers were introduced by Behera and Panda \cite{BP} to be solutions of the Diophantine equation $1+2+\cdots+(n-1) = (n+1)+(n+2)+\cdots+(n+r)$ for some natural number $r$.}. Here $\gcd(B_n, B_{n+2}) = 6$. The main goal of the present paper is to treat the above-mentioned problem systematically. We first show a method to tell which equation is used by two arbitrary positive integers $a$ and $b$, then apply the method to prove the pattern of equations used by various sequences. 

For convenience, we introduce the following notation. Let $\Gamma(a,b)$ be the equation that the pair $(a/d, b/d)$ uses, where $d = \gcd(a, b)$. In particular, $\Gamma(a,b) = 1$ if $(a/d, b/d)$ uses $\eqref{e1}$; otherwise, $\Gamma(a,b) = 2$. For $b/d > 1$, define $\Theta(a, b)$ to be the unique multiplicative inverse of $a/d$ modulo $b/d$ such that $0 < \Theta(a,b) < b/d$. We are ready to state the first result that relates $\Gamma$ to $\Theta$. Note that unlike $\Gamma$, $\Theta(a,b)$ is not necessarily equal to $\Theta(b,a)$.

\begin{theorem}
\label{m1} Let $a,b\in \mathbb{N}$. If $a$ divides $b$ or $b$ divides $a$, then $\Gamma(a, b) = 1$. Otherwise, 
\begin{enumerate}
\item\label{i1} if $a/\gcd(a,b)$ is odd, then $\Gamma(a, b) = 1$ if and only if $\Theta(b, a)$ is odd; 
\item\label{i2} if $a/\gcd(a,b)$ is even, then $\Gamma(a,b) = 1$ if and only if  $\Theta(a,b)$ is odd.
\end{enumerate}
\end{theorem}

\begin{remark}\normalfont
Given $a,b\in \mathbb{N}$ with $b/\gcd(a,b) > 1$, we can find $\Theta(a,b)$ using Euler's Theorem. In particular, Euler's Theorem implies that
$$(a/d)^{\phi(b/d)-1}\times (a/d)\ \equiv\ 1\mod b/d\,,$$
where $d = \gcd(a,b)$ and $\phi$ is the Euler totient function.
Write $(a/d)^{\phi(b/d)-1} = (b/d) \ell + r$ for some $\ell\ge 0$ and $0 < r < b/d$. Then $\Theta(a,b) = r$. 

For example, let $a = 15$ and $b = 85$. We have $d = \gcd(15,85) = 5$, $\phi(b/d) = 16$, and $(a/d)^{\phi(b/d)-1} = 3^{15}$. Since $3^{15} = 17\times 844053 + 6$, we obtain $\Theta(a,b) = 6$. 
\end{remark}

Given a sequence $(a_n)_n\subset\mathbb{N}$, let $\Delta((a_n)_n) = (\Gamma(a_n, a_{n+1}))_n$, i.e., $\Delta$ gives the sequence of equations used by consecutive terms of $(a_n)_n$. Inspired by \cite[Theorem 1.4]{C} that for $\Delta((F_n)_n)$, each equation appears in groups of three alternatively, we construct a sequence whose $\Delta$ has each equation appear in groups of $k$ alternatively ($k\in \mathbb{N}$).

\begin{theorem}\label{m2}
Fix $k\in \mathbb{N}$ and let $a_n = (\lceil 2^{n+k-1}/(2^k+1)\rceil)_{n\ge 1}$. Then $\Delta((a_n)_{n})$ is
$$\underbrace{1,\ldots, 1}_{k}, \underbrace{2, \ldots, 2}_k, \underbrace{1, \ldots, 1}_k, \underbrace{2, \ldots, 2}_k, \ldots\,.$$
\end{theorem}

\begin{remark}\normalfont
The case $k=1$ gives us the sequence $(\lceil 2^n/3\rceil)_{n\ge 1}$, which is \seqnum{A005578} in the Online Encyclopedia of Integer Sequences (OEIS) \cite{Sl}. The case $k = 2$ gives $(\lceil 2^{n+1}/5\rceil)_{n\ge 1}$, which appears to be unavailable in the encyclopedia at the time of the current writing.  
\end{remark}

Next, for a fixed $k\in \mathbb{N}$, we investigate $\Delta((n^k)_{n \ge 1})$. When $k=1$, $\Delta(\mathbb{N})$ has the simple form $1,2,1,2,1,2,\ldots$.
Interestingly, for $k > 1$, $\Delta((n^k)_{\ge 1})$ appears not to have a nice pattern among the first few terms. Let us take a look at $\Delta((n^4)_{n \ge 1})$, for example
$$1, 2, 1, 1, 1, 2, 2, 1, 2, 2, 1, 2, 1, 2, 1, 2, 1, 2, 1, 1, 2, 1, 2, 1, 2, 1, 2, 1, 2, 1, 2, 1, 2, 1, \ldots\,;$$
however, we have the following result about the pattern of $\Delta((n^k)_{n\ge 1})$ in the long run.

\begin{theorem}\label{m3}
Fix $k\in \mathbb{N}$. Then $\Delta((n^k)_{n})$ is eventually $1,2,1,2,1,2, \ldots$. In particular, define 
$$g(x)\ =\ \left(\sum_{i=1}^k x^{i-1}\right)^k\mod x^k\,.$$
When $k$ is odd, let $M_k$ be the smallest positive integer such that 
$$0 \ < \ n^k - g(n) \ <\ n^k\mbox{ and }0\ <\ g(-n)\ <\ n^k,\mbox{ for all }n\ge M_k\,;$$
when $k$ is even, let $M_k$ be the smallest positive integer such that
$$0\ <\ n^k+g(-n) \ <\ n^k\mbox{ and }0 \ <\ g(n) \ <\ n^k, \mbox{ for all }n\ge M_k\,.$$
Then the sequence $(\Gamma(n^k, (n+1)^k))_{n\ge 1}$ starts to be $1,2,1,2,1,2,\ldots$ at $n \le M_k+1$. 
\end{theorem}

The next natural sequences to consider are arithmetic progressions of the form $(a + (n-1)r)_{n\ge 1}$ for fixed $a, r\in \mathbb{N}$. We do not consider geometric progressions with an integral multiplier because every given term in the sequence divides the next term; hence, the sequence $\Delta$ is constantly $1$ according to Theorem \ref{m1}.  We instead investigate $\Delta((a_n)_{n\ge 1})$, where $(a_n)_n$ is a shifted geometric sequence, i.e., $a_n  = ar^{n-1}+1$ for some $a, r\in \mathbb{N}$. 

\begin{theorem}\label{m5}
Let $(a_n)_{n\ge 1}$ be an arithmetic progression. Then $\Delta((a_n)_{n})$ is either $1, 2, 1, 2, \ldots$ or $2, 1, 2, 1, \ldots$.
\end{theorem}

\begin{theorem}\label{m6}
Let $a, r\in \mathbb{N}$ with $r \ge 2$. For each $n\ge 1$, define the sequence $(a_n)_n = (ar^{n-1}+1)_n$. Set $d = \gcd(a + 1, r - 1)$. 
\begin{enumerate}
\item Suppose that $d$ is odd. If $a+1$ does not divide $r-1$, then $\Delta((a_n)_{n\ge 1})$ is constant. If $a+1$ divides $r-1$, then $\Delta((a_n)_{n\ge 1})$ has the first term be $1$ and the later terms be $2$. 
\item If $d$ is even, then $\Delta((a_n)_{n\ge 2})$ alternates between $1$ and $2$.
\end{enumerate}
\end{theorem}

What we have done so far is to determine $\Delta((a_n)_n)$, for a fixed sequence $(a_n)_n$. Note that $\Delta((a_n)_n)$ is the sequence $(\Gamma(a_n, a_{n+1}))_{n\ge 1}$, so both inputs of the function $\Gamma$ change as we move along the sequence.  Our final result instead considers the case when one of the parameters is fixed. Specifically, we study properties of the sequence $(\Gamma(k,n))_{n\ge 1}$ when $k\in \mathbb{N}$ is given. Given a sequence $(a_n)_n$, let $T$ be the smallest positive integer (if any) such that
$a_n = a_{n+T}$ for all $n\ge 1$; the number $T$ is called the period of $(a_n)_n$.

\begin{theorem}\label{m4}
Fix $k\in \mathbb{N}$. For $n\ge 1$, let $a_n = \Gamma(k, n)$. The following hold.
\begin{enumerate}
\item If $k$ is odd, $(a_n)_n$ has period $k$; furthermore, in each period, the number of $1$'s is one more than the number of $2$'s.
\item If $k$ is even, $(a_n)_n$ has period $2k$; furthermore, in each period, the number of $1$'s is two more than the number of $2$'s.
\end{enumerate}
\end{theorem}

Our paper is structured as follows. In Sect. \ref{s2}, we prove Theorem \ref{m1}, which gives a way to compute $\Gamma{(a, b)}$ using $\Theta{(a, b)}$, and provide some preliminary results for our later proofs. In Sect. \ref{s3}, we use
the framework developed in Sect. \ref{s2} to investigate $\Delta{((a_n)_n)}$ for some special sequences $(a_n)_n$. In Sect. \ref{s4}, we prove Theorem \ref{m4}, which describes the period of the sequence $(\Gamma{(k, n)_n})_{n\geq1}$ for any fixed $k\in \mathbb{N}$. In Sect. \ref{s5}, we study sequences satisfying certain linear recurrence relations of order two, thus extending \cite[Theorem 1.4]{C}.

\section{On $\Gamma(a,b)$ and Preliminary Results}\label{s2}

The main goal of this section is to prove Theorem \ref{m1}, which tells us which equation is used by two arbitrary numbers $a, b$, given the parity of $a/\gcd(a,b)$ and of $\Theta$. The key idea is to use the modulo argument to express $\Theta$ in terms of the solution of \eqref{e1} or \eqref{e2}. We then present other useful results to be used in later sections. 

\begin{proof}[Proof of Theorem \ref{m1}] 
We can rewrite \eqref{e1} and \eqref{e2} as 
\begin{equation}\label{e3}a(2x+1)+b(2y+1)\ =\ ab+\ell\,,\end{equation}
where $\ell = 1$ and $-1$, respectively.

If $a$ divides $b$, then $\Gamma(a,b) = \Gamma(1, b/a)$. Note that $(0,0)$ is the solution of
$$1\times (2x+1)+\frac{b}{a}\times (2y+1)\ =\ 1\times \frac{b}{a} +1\,.$$
Hence, $\Gamma(a,b) = 1$. Similarly, if $b$ divides $a$, $\Gamma(a,b) = 1$. 

Next, we prove Item \eqref{i1}. Let $d = \gcd(a,b)$. Suppose that $d\notin \{a, b\}$ and $a/d$ is odd. Assuming $\Gamma(a,b) = 1$, we need to show $\Theta(b,a)$ is odd. By \eqref{e3}, there exist nonnegative integers $x$ and $y$ such that 
\begin{equation}\label{e80}\frac{a}{d}(2x+1) + \frac{b}{d}(2y+1) \ =\ \frac{ab}{d^2}+1\,.\end{equation}
Taking modulo $a/d$ gives
$$\frac{b}{d}(2y+1)\ \equiv\ 1\mod \frac{a}{d}\,.$$
Hence,
$$\Theta(b,a)\frac{b}{d}(2y+1)\ \equiv\ \Theta(b,a)\mod \frac{a}{d}\,.$$
By definition, $\Theta(b,a)b/d \equiv 1\mod a/d$; therefore,
\begin{equation}\label{e4}\Theta(b,a)\ \equiv\ 2y+1\mod \frac{a}{d}\,.\end{equation}
We claim that $\Theta(b,a) = 2y+1$. To do so, it suffices to verify that $0 < 2y+1 < a/d$. Multiplying both sides of \eqref{e80} by $d/b$, we have
\begin{equation}\label{e6}\frac{a}{b}(2x+1) + 2y+1 \ =\ \frac{a}{d}+\frac{d}{b}\,.\end{equation}
Since $d < a$, \eqref{e6} implies that
$$\frac{d}{b}(2x+1)+2y+1 \ <\ \frac{a}{d}+\frac{d}{b}\,,$$
which gives
$$0 \ <\ 2y+1 \ <\ \frac{a}{d}\,,$$
as desired. 

Conversely, assuming that $\Gamma(a,b) = 2$, we prove that $\Theta(b,a)$ is even. By \eqref{e3}, there exist nonnegative integers $x$ and $y$ such that 
\begin{equation}\label{e7}\frac{a}{d}(2x+1) + \frac{b}{d}(2y+1) \ =\ \frac{ab}{d^2}-1\,.\end{equation}
Taking modulo $a/d$ then multiplying both sides by $\Theta(b,a)$ gives
$$\Theta(b,a)\frac{b}{d}(2y+1)\ \equiv\ -\Theta(b,a)\mod \frac{a}{d}\,.$$
Hence,
\begin{equation}\label{e8}2y+1\ \equiv\ \frac{a}{d}-\Theta(b,a)\mod \frac{a}{d}\,.\end{equation}
It is easily seen from \eqref{e7} that $0< 2y+1 < a/d$. It follows from the definition of $\Theta$ that $0 < a/d - \Theta(b,a) < a/d$. Therefore, \eqref{e8} implies that $2y + 1 = a/d - \Theta(b,a)$.  Since $a/d$ is odd, $\Theta(b,a)$ must be even. This completes our proof of Item \eqref{i1}.

It remains to prove Item \eqref{i2}, but \eqref{i2} follows directly from Item \eqref{i1}. Indeed, suppose that $a/d$ is even. Since $\gcd(a/d, b/d) = 1$, we know that $b/d$ is odd. By Item \eqref{i1}, $\Gamma(a,b) = 1$ if and only if $\Theta(a,b)$ is odd. 
\end{proof}

The next two results are used in the proof of Theorem \ref{m2}. 

\begin{corollary}\label{c1}
For $a\ge 1$, $\Gamma(a, 2a) = 1$; for $a\ge 2$, $\Gamma(a, 2a-1) = 2$.
\end{corollary}
\begin{proof}
By Theorem \ref{m1}, $\Gamma(a,2a) = 1$. To see that $\Gamma(a, 2a-1) = 2$, we first compute $\Theta(2a-1, a) = a-1$ and $\Theta(a, 2a-1) = 2$. If $a$ is odd, then $\Theta(2a-1,a)$ is even; Theorem \ref{m1} states that $\Gamma(a,2a-1) = 2$. If $a$ is even, then $\Gamma(a, 2a-1) = 2$ because $\Theta(a, 2a-1)$ is even.
\end{proof}

\begin{lemma}\label{l1}
Let $k, n\in \mathbb{N}$ and $m\in \{1, \ldots, 2k\}$ such that $n\equiv m\mod 2k$. Then
$$2^{n+k-1}\ \equiv\ \begin{cases}-2^{m-1}\mod 2^k+1\,, &{\rm if } \mbox{ } 1\le m\le k\,,\\ 2^{m-k-1}\mod 2^k+1\,, &{\rm if } \mbox{ }k+1\le m\le 2k\,.\end{cases}$$
\end{lemma}

\begin{proof}
Write $n = 2kj+m$ for some $j\ge 0$. 

If $1\le m\le k$, then
\begin{equation*}
2^{n+k-1}-2^{m+k-1}\ =\ 2^{2kj+m+k-1}-2^{m+k-1}\ =\ 2^{m+k-1}(2^{2kj}-1)\,,
\end{equation*}
which is divisible by $2^k+1$. Hence, $2^{n+k-1}\equiv 2^{m+k-1}\equiv -2^{m-1}\mod 2^k+1$.

If $k+1\le m\le 2k$, then 
\begin{equation*}
2^{n+k-1}-2^{m-k-1}\ =\ 2^{2kj+2k+m-k-1}-2^{m-k-1}\ =\ 2^{m-k-1}(2^{2k(j+1)}-1)\,,
\end{equation*}
which is divisible by $2^k+1$. Hence, $2^{n+k-1}\equiv 2^{m-k-1}\mod 2^k+1$.
\end{proof}

Our next two lemmas show that for arithmetic progressions and shifted geometric progressions $(a_n)_n$, $\gcd(a_n, a_{n+1})$ is a constant. 

\begin{lemma}\label{lem_ap}
Fix $a, r\in \mathbb{N}$. Let $a_n = a+(n-1)r$. Then $\gcd(a_n,a_{n+1}) = \gcd(a, r)$.
\end{lemma}

\begin{proof}
 Let $n\in \mathbb{N}$. We have that $\gcd(a_n, a_{n+1})$ divides $2a_{n+1}-a_n = a_{n+2}$. Hence, $\gcd(a_n, a_{n+1})$ divides $\gcd(a_{n+1}, a_{n+2})$. Conversely, $\gcd(a_{n+1}, a_{n+2})$ divides $2a_{n+1}-a_{n+2} = a_n$. Hence, $\gcd(a_{n+1}, a_{n+2})$ divides $\gcd(a_n, a_{n+1})$. Consequently, $$\gcd(a_n, a_{n+1}) \ =\ \gcd(a_{n+1}, a_{n+2})\,.$$ Therefore, for all $n$, $\gcd(a_n, a_{n+1}) = \gcd(a_1, a_2) = \gcd(a, a+r) = \gcd(a,r)$. 
\end{proof}
\begin{lemma}\label{l4}
Fix $a,r\in \mathbb{N}$ with $r\ge 2$. Let $a_n = ar^{n-1}+1$. Then $\gcd(a_n,a_{n+1})=\gcd(a+1, r-1)$ for all $n\in\mathbb{N}$.
\end{lemma}
\begin{proof}
Let $k\ge 0$. We shall show that 
\begin{equation}\label{e60}\gcd(ar^{k}+1, r-1)\ =\ \gcd(a+1, r-1)\,.\end{equation}
The equality clearly holds when $k = 0$. Suppose that it holds for some $k = \ell\ge 0$. We have
\begin{align*}\gcd(ar^{\ell+1}+1, r-1)&\ =\ \gcd(r(ar^{\ell}+1), r-1)\\
&\ =\ \gcd(ar^{\ell}+1, r-1)\ =\ \gcd(a+1, r-1)\,.\end{align*}
By mathematical induction, we are done. 

For $n\in \mathbb{N}$, by \eqref{e60}, we have
\begin{align*}
  \gcd(a_n,a_{n+1}) & \ = \ \gcd(ar^{n-1}+1,ar^n+1)\ = \ \gcd(ar^{n-1}+1,ar^{n-1}(r-1))\\
  &\ =\ \gcd(ar^{n-1}+1, r-1)\ =\ \gcd(a+1, r-1)\,.
\end{align*}
\end{proof}

\section{The Sequence $\Delta((a_n)_n)$ for Some Special $(a_n)_n$}\label{s3}
In this section, we look at various sequences $(a_n)_n$ and determine their $\Delta((a_n)_n)$. Results from Sect. \ref{s2} will be used in due course. 

\subsection{Sequences Whose $\Delta$ Has a Periodic Form}
\begin{proof}[Proof of Theorem \ref{m2}]
Fix $k\in \mathbb{N}$. Let $a_n = \lceil 2^{n+k-1}/2^{k+1}\rceil$. Thanks to Corollary \ref{c1}, it suffices to prove that $a_1 = 1$ and for $n\ge 1$,
$$a_{n+1}\ =\ \begin{cases}2a_n\,,&\mbox{ if }n\equiv 1\,, \ldots, k\mod 2k\,,\\ 2a_n-1\,,&\mbox{ if }n\equiv k+1, \ldots, 2k\mod 2k\,.\end{cases}$$
Fix $n\in \mathbb{N}$ and pick $m\in \{1,\ldots, 2k\}$ such that $n\equiv m\mod 2k$. 

\mbox{ }

Case 1: $1\le m\le k$. By Lemma \ref{l1}, $2^{n+k-1}\equiv 2^{k}+1-2^{m-1}\mod 2^k+1$, so
\begin{align*}
2^{n+k-1}&\ =\ (2^k+1)\left\lfloor \frac{2^{n+k-1}}{2^k+1}\right\rfloor + 2^k+1-2^{m-1}\\
&\ =\ (2^k+1)\left(\left\lceil \frac{2^{n+k-1}}{2^k+1}\right\rceil-1\right)+ 2^k+1-2^{m-1}\,.
\end{align*}
Multiplying both sides by $2$ gives
$$2^{n+k}\ =\ 2(2^k+1)\left\lceil \frac{2^{n+k-1}}{2^k+1}\right\rceil-2^{m}\,.$$
Hence,
$$\frac{2^{n+k}}{2^k+1}\ =\ 2\left\lceil \frac{2^{n+k-1}}{2^k+1}\right\rceil - \frac{2^m}{2^k+1}\,.$$
Therefore, 
$$\left\lceil \frac{2^{n+k}}{2^k+1}\right\rceil\ =\ 2\left\lceil \frac{2^{n+k-1}}{2^k+1}\right\rceil\,,$$
i.e., $a_{n+1} = 2a_{n}$. 

\mbox{ }

Case 2: $k+1\le m\le 2k$. By Lemma \ref{l1}, $2^{n+k-1}\equiv 2^{m-k-1}\mod 2^k+1$, so
\begin{align*}
2^{n+k-1}&\ =\ (2^k+1)\left\lfloor \frac{2^{n+k-1}}{2^k+1}\right\rfloor + 2^{m-k-1}\\
&\ =\ (2^k+1)\left(\left\lceil \frac{2^{n+k-1}}{2^k+1}\right\rceil-1\right)+ 2^{m-k-1}.
\end{align*}
Multiplying both sides by $2$ gives
$$2^{n+k}\ =\ (2^k+1)\left(2\left\lceil \frac{2^{n+k-1}}{2^k+1}\right\rceil-1\right) - (2^k+1- 2^{m-k}).$$
Hence,
$$\left\lceil\frac{2^{n+k}}{2^k+1}\right\rceil \ =\ 2\left\lceil \frac{2^{n+k-1}}{2^k+1}\right\rceil - 1\,,$$
i.e., $a_{n+1} = 2a_n - 1$. This completes our proof. 
\end{proof}

\subsection{The $k$\textsuperscript{th}-Power Sequence}
\begin{proof}[Proof of Theorem \ref{m3}]
Let us first assume that $k$ is odd. Choose $n\in \mathbb{N}_{\ge 2}$.
Let $x = \Theta((n-1)^k, n^k)$ and $y = \Theta((n+1)^k, n^k)$.
Since $(n-1)^k x\equiv 1\mod n^k$, we have, in modulo $n^k$,
\begin{equation}\label{e16}
(-1)^k x\ \equiv\ (n^k - 1)^k x\ =\  \left(\sum_{i=1}^k n^{i-1}\right)^k (n-1)^k x \ \equiv\ \left(\sum_{i=1}^k n^{i-1}\right)^k\,.
\end{equation}
As $k$ is odd, it follows that
\begin{equation}\label{e10}
    x\ \equiv\ n^k-\left(\sum_{i=1}^k n^{i-1}\right)^k\mod n^k\,.
\end{equation}
Since $(n+1)^k y\equiv 1\mod n^k$, we have, in modulo $n^k$,
\begin{equation}\label{e11}
y\ \equiv\ (n^k+1)^k y\ =\ \left(\sum_{i=1}^k (-1)^{i-1}n^{i-1}\right)^k (n+1)^k y\ \equiv\ \left(\sum_{i=1}^k (-1)^{i-1}n^{i-1}\right)^k\,.
\end{equation}

Define $u(x)= \left(\sum_{i=1}^k x^{i-1}\right)^k$ and $v(x) = \left(\sum_{i=1}^k (-1)^{i-1}x^{i-1}\right)^k$. Since $u(x) + v(x) = u(-x) + v(-x)$, $u(x)+v(x)$ is an even function; therefore, the coefficients of odd powers in $u(x)$ are equal to the negative of the corresponding coefficients of $v(x)$. On the other hand, $u(x)-v(x)$ is an odd function, so the coefficients of even powers in $u(x)$ are equal to the corresponding even powers in $v(x)$. Therefore, if we let $g(x)$ be the tail of $u(x)$ up to the power $k-1$, then $g(-x)$ is the tail of $v(x)$ up to the power $k-1$. From \eqref{e10} and \eqref{e11}, we have
\begin{equation}\label{e12}x\ \equiv\ n^k-g(n)\mbox{ and }y\ \equiv\ g(-n)\mod n^k\,.\end{equation}
Since $\deg(g(n)) = \deg(g(-n)) = k-1$, we can choose $N_1\in \mathbb{N}$ such that 
$n^{k}  > g(n)$,  for all $n \ge N_1$. 
Observe that the coefficient of $x^{k-1}$ is positive in both $g(x)$ and $g(-x)$; hence, we can choose $N_2\in \mathbb{N}$ such that 
$g(n), g(-n) > 0$ whenever $n \ge N_2$. Set $N = \max\{N_1, N_2\}$ to have
\begin{equation}\label{e13}0 \ < \ n^k - g(n) \ <\ n^k\mbox{ and }0\ <\ g(-n)\ <\ n^k\,.\end{equation}
From \eqref{e12} and \eqref{e13}, we obtain
\begin{equation}\label{e15}x \ =\ n^k-g(n)\mbox{ and }y \ =\ g(-n)\,,\quad \mbox{ for all }n\ge N\,.\end{equation}
Therefore, 
\begin{equation}\label{e14}x + y \ =\  n^k-(g(n)-g(-n))\,, \quad \mbox{ for all }n\ge N\,.\end{equation}
As discussed above, all coefficients of $g(x)-g(-x)$ are even. Hence, the parity of $x+y$ is the same as the parity of $n^k$.

Taking an even $M \ge N-1$, we show that
$$\Gamma(M^k, (M+1)^{k}) \ \neq\ \Gamma((M+1)^{k}, (M+2)^{k})\,.$$
To do so, we invoke Theorem \ref{m1}. By Theorem \ref{m1}, $\Gamma(M^k, (M+1)^k) = 1$ if and only if $\Theta(M^k, (M+1)^k)$ is odd; furthermore, $\Gamma((M+1)^k, (M+2)^k) = 1$ if and only if $\Theta((M+2)^k, (M+1)^k)$ is odd. By \eqref{e14}, $\Theta(M^k, (M+1)^k)+\Theta((M+2)^k, (M+1)^k)$ has the same parity as $(M+1)^k$, which is odd. Hence, $\Theta(M^k, (M+1)^k)$ and $\Theta((M+2)^k, (M+1)^k)$ have different parities and so, $\Gamma(M^k, (M+1)^{k})  \neq \Gamma((M+1)^{k}, (M+2)^{k})$.

To show that $\Gamma$ alternates between $1$ and $2$, it remains to verify that 
$\Theta(M^k, (M+1)^k)$ and  $\Theta((M+2)^k, (M+3)^k)$ have the same parity. According to \eqref{e15},
\begin{align*}
\Theta(M^k, (M+1)^k) &\ =\ (M+1)^k-g(M+1)\,,\\
\Theta((M+2)^k, (M+3)^k)&\ =\ (M+3)^k-g(M+3)\,.
\end{align*}
Hence, $\Theta((M+2)^k, (M+3)^k)-\Theta(M^k, (M+1)^k) = (M+3)^k-(M+1)^k-(g(M+3)-g(M+1))$, which is clearly even. It follows that  $\Theta((M+2)^k, (M+3)^k)$ and $\Theta(M^k, (M+1)^k)$ have the same parity. Invoking Theorem \ref{m1}, we finish the proof when $k$ is odd. 
\end{proof}

\begin{proof}[Proof of Theorem \ref{m3} for even $k$]
Let $n, x, y$, and $g(x)$ be chosen as in the proof of Theorem \ref{m3} for odd $k$. From \eqref{e16} and that $k$ is even, we know that
$$
x\ \equiv\ g(n)\mod n^k\,.
$$
On the other hand, in modulo $n^k$, 
$$y \ =\ (-1)^k y\ \equiv\ (n^k-1)^k y\ =\ \left(\sum_{i=1}^k(-1)^in^{i-1}\right)^k(n+1)^k y\ \equiv\ \left(\sum_{i=1}^k(-1)^in^{i-1}\right)^k\,.$$
Since $k$ is even,
$$y\ \equiv\ \left(\sum_{i=1}^k(-1)^{i-1}n^{i-1}\right)^k\mod n^k\,; \mbox{ hence, }y\ \equiv \ n^k + g(-n)\mod n^k\,.$$
Choose $M\in \mathbb{N}$ such that $0 < n^k+g(-n) < n^k$ and $0 < g(n) < n^k$ whenever $n\ge M$. This can be done since $\deg(g(n)) = \deg(g(-n)) = k-1$, and the coefficients of $n^{k-1}$ in $g(n)$ and $g(-n)$ are positive and negative, respectively. Therefore, for $n\ge M$, $x = g(n)$ and $y = n^k + g(-n)$. That $x+y = n^k + (g(n)+g(-n))$ shows that $x+y$ has the same parity as $n^k$. Using the same argument as in the proof of Theorem \ref{m3} for odd $k$ completes the proof. 
\end{proof}

\subsection{Arithmetic Progression}

\begin{proof}[Proof of Theorem \ref{m5}]
For $a_n\ =\ a\ +\ r(n-1)$ , using Lemma \ref{lem_ap}, we let 
    $\gcd(a_n,a_{n+1}) = d$ for all $n\in\mathbb{N}$. Fix $n\ge 2$ and consider the consecutive terms $a_n, a_n+r, a_n+2r$ in the sequence. Note that $a_n\nmid a_{n+1}$ because $n\ge 2$. Our goal is to show that $\Gamma(a_n, a_{n+2}) \neq \Gamma(a_{n+1}, a_{n+2})$.

\mbox{ }
    
    Case 1: $a_{n+1}/d$ is odd.
    Let $x = \theta(a_n,a_{n+1})$ and $y = \theta(a_{n+2},a_{n+1})$. Then
    \begin{align*}\frac{a_n}{d}x &\ \equiv \ 1 \mod \frac{a_n}{d}+\frac{r}{d}\,,\\
     \left(\frac{a_n}{d}+\frac{2r}{d}\right)y &\ \equiv\ 1 \mod \frac{a_n}{d}+\frac{r}{d}\,.\end{align*}
These imply that 
    \begin{align}\label{e81'}\frac{r}{d}\left(\frac{a_n}{d}+\frac{r}{d}-x\right) &\ \equiv\ 1 \mod \frac{a_n}{d}+\frac{r}{d}\,,\\
    \label{e82'}\frac{r}{d}y &\ \equiv\ 1 \mod \frac{a_n}{d}+\frac{r}{d}\,.\end{align}
    Since $0 < x < a_n/d+r/d$, we  get
    \begin{equation}\label{e85'} 0 \ <\ \frac{a_n}{d}+\frac{r}{d}-x \ <\ \frac{a_n}{d}+\frac{r}{d}\,. \end{equation}
    From \eqref{e81'}, \eqref{e82'}, and \eqref{e85'}, we deduce that
    $$y\ =\ \frac{a_n}{d}+\frac{r}{d} - x\,.$$
    Hence, $x+y = (a_n+r)/d$. That $\gcd(a_n/d, r/d) = \gcd(a_n/d, (a_n+r)/d) = 1$ implies that $(a_n+r)/d$ is odd. Consequently, $x$ and $y$ have different parities. Using Theorem \ref{m1}, we conclude that $\Gamma(a_n, a_{n+1})\neq \Gamma(a_{n+1}, a_{n+2})$.

\mbox{ }
    
    Case 2: $a_{n+1}/d$ is even, then $a_n/d$ and $a_n/d+ 2r/d$ are both odd. Let $x = \theta(a_{n+1},a_n)$ and $y = \theta(a_{n+1},a_{n+2})$.
  Then 
  \begin{align*}\left(\frac{a_n}{d}+\frac{r}{d}\right)x &\ \equiv \ 1 \mod \frac{a_n}{d}\,,\\
 \left(\frac{a_n}{d}+\frac{r}{d}\right)y &\ \equiv\ 1 \mod \frac{a_n}{d}+\frac{2r}{d}\,.\end{align*}
Equivalently, there exist positive integers $k_1,k_2< (a_n+r)/d$ such that
\begin{align}
\label{e92}\left(\frac{a_n}{d}+\frac{r}{d}\right)x &\ = \ 1 + k_1\frac{a_n}{d}\,,\\
\label{e91}\left(\frac{a_n}{d}+\frac{r}{d}\right)y &\ = \ 1 + k_2\left(\frac{a_n}{d}+\frac{2r}{d} \right)\,.
\end{align}
Subtracting \eqref{e92} and \eqref{e91} side by side gives
    $$\frac{a_n+r}{d}\left(x-y+2k_2 \right) \ =\  \frac{a_n}{d}\left( k_1+k_2\right).$$
    Since $\gcd((a_n+r)/d,a_n/d) = 1$, $(a_n+r)/d$  divides $k_1+k_2$.
Observe that $0< k_1+k_2 < 2(a_n+r)/d$ because $0<k_1, k_2 < (a_n+r)/d$. Therefore, 
$$k_1+k_2 \ =\ \frac{a_n+r}{d} \mbox{ and }x-y+2k_2 = \frac{a_n}{d}\,.$$
Since $a_n/d$ is odd, $x$ and $y$ must have different parities. We again have $\Gamma(a_n,a_{n+1}) \ne \Gamma(a_{n+1},a_{n+2})$.

We have shown that $\Gamma(a_n,a_{n+1}) \ne \Gamma(a_{n+1},a_{n+2})$ for all $n\ge 2$. It remains to show that $\Gamma(a_1, a_2)\neq \Gamma(a_2, a_3)$. If $a_1\nmid a_2$, the same reasoning as above gives $\Gamma(a_1,a_2) \neq \Gamma(a_2,a_3)$.
If $a_1\mid a_2$, $\Gamma(a_1,a_2)\ =\ 1$. Let $a_2\ =\ pa_1\ =\ pa$ for some $p\ge 2$. Then $a_3\ =\ 2a_2-a_1\ =\ (2p-1)a$ and $$\Gamma(a_2,a_3)\ =\ \Gamma(pa, (2p-1)a)\ =\ \Gamma(p, 2p-1)\,.$$
By Corollary \ref{c1}, $\Gamma(a_2,a_3)\ =\ 2 \neq\Gamma(a_1,a_2)$.
\end{proof}

\subsection{Shifted Geometric Progression}

\begin{proof}[Proof of Theorem \ref{m6} when $d$ is odd]
 For $n\in \mathbb{N}$, let $b_n = ar^{n-1}$. Then we can write the three consecutive terms $a_n,a_{n+1},a_{n+2}$ as 
$$b_n+1 , rb_n+1 , r^2b_n+1\,.$$
 Since $d$ is odd, $rb_n+1$ is odd. Indeed, suppose otherwise that $rb_n+1$ is even. Then both $r$ and $b_n$ are odd. Hence, $r^2b_n+1$ is even, thus $\gcd(a_{n+1}, a_{n+2}) = \gcd(rb_n+1, r^2b_n+1)$ is even. This contradicts Lemma \ref{l4} that $\gcd(a_{n+1}, a_{n+2}) = d$, which is odd. 

Observe that $rb_n+1$ does not divide $r^2b_n+1$ because
$$(rb_n+1)(r-1)\ =\ r^2b_n+r-ar^n-1\ <\ r^2b_n+1\ <\ (rb_n+1)r\,.$$
Let $x = \Theta(b_n+1, rb_n+1)$ and $y = \Theta(r^2b_n+1, rb_n+1)$.
Then
\begin{align}\label{e70}\frac{(b_n+1)x}{d} &\ \equiv \ 1 \mod \frac{rb_n+1}{d}\,,\\
\label{e71}\frac{(r^2b_n+1)y}{d} &\ \equiv \ 1 \mod \frac{rb_n+1}{d}\,.\end{align}
Note that \eqref{e71} is equivalent to
\begin{equation}\label{e72}\frac{(1-r)y}{d} \ \equiv \ 1 \mod \frac{rb_n+1}{d}\,. \end{equation}
Multiplying \eqref{e72} by $b_n$ gives
\begin{equation*}\frac{(b_n+1)y}{d} \ \equiv\ b_n \mod \frac{rb_n+1}{d}\,,\end{equation*}
which implies that
\begin{equation}\label{e73}\frac{(b_n+1)((rb_n+1)/d+d-y)}{d} \ \equiv \ 1 \mod \frac{rb_n+1}{d}\,.\end{equation}

\mbox{ }

Case 1: $a+1$ does not divide $r-1$. Then $d < a+1$, so $d \le b_1$. We have $b_n\ge d$ for all $n\in \mathbb{N}$.  
By \eqref{e72},
$(rb_n+1)/d$ divides $(r-1)y/d + 1$; hence,
$$\frac{(r-1)y}{d}\ \ge\ \frac{rb_n+1}{d}-1\ >\ r-1\,,$$
where the second inequality is due to $b_n + 1/r > d$. As a result, $y > d$. Combining with the definition of $y$, we obtain 
\begin{equation}\label{e74} 0\ <\ \frac{rb_n+1}{d}+d-y \ <\ \frac{rb_n+1}{d}\,.\end{equation}
It follows from \eqref{e70}, \eqref{e73}, and \eqref{e74} that
$$x+y\ =\ \frac{rb_n+1}{d}+d\,.$$
Since both $rb_n+1$ and $d$ are odd, $x$ and $y$ have the same parity. 

We now apply Theorem \ref{m1} Item \eqref{i1}. We claim that $b_n+1$ does not divide $rb_n+1$. Indeed, when $n = 1$, we have $b_1 + 1 = a+1$, while $rb_1+1 = ar+1$. If $a+1$ divides $ar+1$, then $a+1$ divides $a(r-1)$. This contradicts our assumption that $a+1$ does not divide $r-1$. For all $n\ge 2$,
$$(r-1)(b_n+1)\ =\ rb_n - ar^{n-1} + r - 1\ <\ rb_n+1\ <\ r(b_n+1)\,.$$
The condition $n\ge 2$ is used in claiming the first inequality.
Hence, $\Gamma(rb_n+1, b_n+1)$ is determined by the parity of $x$. Similarly, $\Gamma(rb_n+1, r^2b_n+1)$ is determined by the parity of $y$. Since $x$ and $y$ have the same parity, $\Gamma(a_n, a_{n+1}) = \Gamma(a_{n+1}, a_{n+2})$.

\mbox{ }

Case 2: $a+1$ divides $r-1$. Then $d = a+1$. We have $b_n = ar^{n-1}\ge d$ for all $n\ge 2$. The same argument as in Case 1 shows that $\Gamma(a_n, a_{n+1}) = \Gamma(a_{n+1}, a_{n+2})$ for all $n\ge 2$. By Lemma \ref{l4}, $\gcd(a+1, ar+1) = \gcd(a+1, r-1) = a+1$. Hence, $$\Gamma(a_1, a_2)\ =\ \Gamma(a+1, ar+1)\ =\ \Gamma(1, (ar+1)/(a+1))\ =\ 1\,.$$
It remains to verify that $\Gamma(a_2, a_3) = \Gamma(ar+1, ar^2+1) = 2$. This follows directly from the observation that the equation
$$\frac{ar+1}{a+1}x + \frac{ar^2+1}{a+1}y + 1\ =\ \frac{1}{2}\left(\frac{ar+1}{a+1}-1\right)\left(\frac{ar^2+1}{a+1}-1\right)$$
has solution
\begin{equation}\label{e81}(x,y)\ =\ \left(\frac{ar}{2}-1, \frac{a}{2}\left(\frac{r-1}{a+1}-1\right)\right)\,.\end{equation}
\end{proof}

\begin{proof}[Proof of Theorem \ref{m6} when $d$ is even]
For $n\in \mathbb{N}_{\ge 2}$, let $b_n = ar^{n-1}$. Then we can write the three consecutive terms $a_n,a_{n+1},a_{n+2}$ as 
 $$b_n+1, rb_n+1,r^2b_n+1\,.$$
Our goal is to show that $\Gamma(a_n, a_{n+1})\neq \Gamma(a_{n+1}, a_{n+2})$.

\mbox{ }

Case 1: $(rb_n+1)/d$ is odd. Let $u = \Theta(b_n+1, rb_n+1)$ and $v = \Theta(r^2b_n+1, rb_n+1)$. For $n\ge 2$, as in the proof for odd $d$,
\begin{equation*}u+v\ =\ \frac{rb_n+1}{d}+d\,,\end{equation*}
 which implies that $x$ and $y$ have different parities. Hence, $\Gamma(a_n, a_{n+1})\neq\Gamma(a_{n+1}, a_{n+2})$.

\mbox{ }

Case 2: $(rb_n+1)/d$ is even. It holds that $b_n+1$ does not divide $rb_n+1$. 
Indeed, for all $n\ge 2$,
$$(r-1)(b_n+1)\ =\ rb_n - ar^{n-1} + r - 1\ <\ rb_n+1\ <\ r(b_n+1)\,.$$
(The condition $n\ge 2$ is used in claiming the first inequality.) Let $x = \Theta(rb_n+1, b_n+1)$ and $y = \Theta(rb_n+1, r^2b_n+1)$. Let $k_1, k_2 > 0$ such that
\begin{align}
\label{e82}\frac{(rb_n+1)x}{d} &\ =\ 1 + \frac{(b_n+1)k_1}{d}\,,\\
\label{e83}\frac{(rb_n+1)y}{d} &\ =\ 1 + \frac{(r^2b_n+1)k_2}{d}\,.
\end{align}
Since $0 < x < (b_n+1)/d$ and $0 < y < (r^2b_n+1)/d$, we know that 
\begin{equation}\label{e88}0 \ < \ k_1, k_2 \ <\ \frac{rb_n+1}{d}\,.\end{equation}

\begin{claim}\label{cl2}
It holds that $(rb_n+1)/d$ divides $k_1+k_2+d$.
\end{claim}
\begin{proof}
It is easy to check that \eqref{e83} is equivalent to
\begin{equation}
\label{e84}\frac{(rb_n+1)(yb_n-rk_2b_n+k_2)}{d}\ =\  \frac{(b_n+1)(k_2+d)}{d} - 1\,.
\end{equation}
Add \eqref{e82} and \eqref{e84} side by side to obtain
\begin{equation}\label{e85}\frac{rb_n+1}{d}(x+yb_n-rk_2b_n+k_2)\ =\ \frac{b_n+1}{d}(k_1+k_2+d)\,.\end{equation}
Since $\gcd((rb+1)/d, (b_n+1)/d) = 1$, we have the desired conclusion.
\end{proof}

By the definition of $x$ and \eqref{e82}, we have
$$\frac{k_1(b_n+1)}{d} \ <\ \frac{(rb_n+1)((b_n+1)/d-1)}{d}\,,$$
which gives 
\begin{equation}\label{e86}k_1 \ <\ \frac{rb_n+1}{d}-\frac{rb_n+1}{b_n+1}\,.\end{equation}
Since $n\ge 2$,
\begin{equation}\label{e87}d\ \le\ r-1\ <\ \frac{ar^n+1}{ar^{n-1}+1}\ =\ \frac{rb_n+1}{b_n+1}\,.\end{equation}
By \eqref{e86} and \eqref{e87}, 
\begin{equation}\label{e89}k_1 \ <\ \frac{rb_n+1}{d}-d\,.\end{equation}
It then follows from \eqref{e88} and \eqref{e89} that
$$k_1 + k_2 + d \ <\ 2\frac{rb_n+1}{d}\,,$$
which, together with Claim \ref{cl2}, gives 
$$k_1 + k_2 + d\ =\ \frac{rb_n+1}{d}\,.$$
Hence, by \eqref{e85}, 
\begin{equation}\label{e90}x+yb_n-rk_2b_n+k_2\ =\ \frac{b_n+1}{d}\,.\end{equation}
We determine the parity of each integer in \eqref{e90} as follows.
\begin{itemize}
\item Since $(rb_n+1)/d$ is even and $\gcd((rb_n+1)/d, (b_n+1)/d) = 1$, $(b_n+1)/d$ must be odd. 
\item Since $d$ is even and $(b_n+1)/d$ is odd, $b_n$ must be odd. 
\item The integer $r$ is odd because $d = \gcd(a+1, r-1)$ is even.
\item Finally, because the left side of Equation \eqref{e83} is even, $k_2$ must odd.
\end{itemize}
Therefore, \eqref{e90} guarantees that $x$ and $y$ have different parities. Applying Theorem \ref{m1}, we conclude that 
$\Gamma(a_n, a_{n+1})\neq \Gamma(a_{n+1}, a_{n+2})$. This completes our proof. 
\end{proof}

\section{The Sequence $(\Gamma(k, n))_n$ for a Fixed $k$}\label{s4}
We prove Theorem \ref{m4}; the main ingredient is Theorem \ref{m1}. We split the proof into two cases corresponding to the parity of $k$, where the case $k$ is even is more technically involved. The two following lemmas are useful in proving the period. The proof of the following lemma is well-known and can be found, for example, at \url{https://math.stackexchange.com/questions/3876246}.

\begin{lemma}\label{l2} Suppose that $(a_n)_{n\ge 1}$ is periodic, and there exists $k\in \mathbb{N}$ with the property $a_n = a_{n+k}$ for all $n\ge 1$. If $T$ is the period of the sequence $(a_n)_{n\ge 1}$, then $T$ divides $k$. 
\end{lemma}

\begin{lemma}\label{l3} 
The sequence $\Delta(\mathbb{N})$ is $1, 2, 1, 2, 1, 2, \ldots$. 
\end{lemma}
\begin{proof}
The sequence $\Delta(\mathbb{N})$ starts with $1$ because $\Gamma(1, 2) = 1$. Apply Theorem \ref{m5} and we are done. 
\end{proof}

\begin{proof}[Proof of Theorem \ref{m4} when $k$ is odd] The case $k = 1$ is trivial, so we assume that $k \ge 3$.
Let $u, v\in \mathbb{N}$ such that $v = u+k$. We show that $\Gamma(k, u) = \Gamma(k, v)$. 

\mbox{ }

Case 1: $u$ divides $k$, giving $\Gamma(k, u) = 1$. Write $k = u\ell$ for some odd $\ell\in \mathbb{N}$. Then $v = u + k = u(\ell+1)$. We have
$\Gamma(k, v) = \Gamma(u\ell, u(\ell+1)) = \Gamma(\ell, \ell+1)$. Since 
$$\ell \frac{\ell-1}{2} + (\ell+1)\times 0\ =\ \frac{(\ell-1)\ell}{2}\,,$$
$\Gamma(\ell, \ell+1) = 1$; hence, $\Gamma(k, u) = \Gamma(k, v) = 1$. 

\mbox{ }

Case 2: $k$ divides $u$. Then $k$ divides $v$, and $\Gamma(k, u) = \Gamma(k, v) = 1$.

\mbox{ }

Case 3: $u$ does not divide $k$, and $k$ does not divide $u$. Let $d = \gcd(k, u)$. Since $k$ is odd, $k/d$ is odd. By Theorem \ref{m1}, $\Gamma(k, u)$ is determined by the parity of $x := \Theta(u, k)$. Similarly, $\Gamma(k, v)$ is determined by the parity of $y:= \Theta(v, k)$. Hence, it suffices to show that $\Theta(u, k)$ and $\Theta(v, k)$ have the same parity. By definition,
\begin{align}
&\label{e30}\frac{ux}{d}\ \equiv\ 1\mod \frac{k}{d}\,,\\
&\label{e31}\frac{vy}{d}\ \equiv\ 1\mod \frac{k}{d}\,.
\end{align}
Furthermore, that $v\equiv u\mod k$, and \eqref{e31} implies that 
\begin{equation}\label{e32}\frac{uy}{d}\ \equiv\ 1\mod \frac{k}{d}\,.\end{equation}
It follows from \eqref{e30} and \eqref{e32} that $x = y$. This completes our proof that $\Gamma(k, u) = \Gamma(k, v)$ whenever $u\equiv v\mod k$. However, this is not enough to conclude that the period of $(\Gamma(k, n))_{n\ge 1}$ is $k$; in the following, we show that this is indeed the case. 

Consider the first $k$ terms of $(\Gamma(k, n))_{n\ge 1}$. Pick $1\le s \le (k-1)/2$ and let $r = \gcd(k,s) = \gcd(k-s, k)$. 

\mbox{ }

Case I: if $s$ divides $k$, then $\Gamma(k, s) = 1$. Write $k = s\ell$ for some odd $\ell\in \mathbb{N}_{\ge 3}$. By Lemma \ref{l3}, we have
$\Gamma(k, k-s) = \Gamma(\ell, \ell-1) = 2$. Hence, $\Gamma(k,s)\neq \Gamma(k, k-s)$.

\mbox{ }

Case II: suppose that $s$ does not divide $k$. Observe that $k/2 < k-s < k$, so $k-s$ does not divide $k$. By Theorem \ref{m1}, $\Gamma(k, s)$ and $\Gamma(k, k-s)$ are determined by the parity of $\Theta(s, k)$ and $\Theta(k-s, k)$, respectively. By definition,
\begin{align}
\label{e33}\Theta(s, k)\frac{s}{r}&\ \equiv\ 1\mod \frac{k}{r}\,,\\ 
\label{e34}\Theta(k-s, k)\frac{k-s}{r}&\ \equiv\ 1\mod \frac{k}{r} \ \Longrightarrow\  \left(\frac{k}{r}-\Theta(k-s, k)\right)\frac{s}{r}\ \equiv\ 1\mod \frac{k}{r}\,.
\end{align}
It follows from \eqref{e33} and \eqref{e34} that 
$$\Theta(s, k) + \Theta(k-s, k) \ =\ \frac{k}{r}\,.$$
Since $k/r$ is odd, $\Theta(s,k)\not\equiv \Theta(k-s, k)\mod 2$, thus $\Gamma(k, s)\neq \Gamma(k, k-s)$.

We have shown that $\Gamma(k,s) \neq \Gamma(k, k-s)$ for all $1\le s\le (k-1)/2$. Along with the fact that $\Gamma(k, k) = 1$, we know that within the first $k$ terms of $(\Gamma(k, n))_{n\ge 1}$, the number of $1$'s is one more than the number of $2$'s.

We are ready to conclude the proof that $(\Gamma(k, n))_{n\ge 1}$ has period $k$. Let $T$ be the period of $(\Gamma(k, n))_{n\ge 1}$. By Lemma \ref{l2}, $T$ divides $k$. Hence, within the first $k$ terms, there are $k/T$ copies of the period. Let $p$ and $q$ be the number of $1$'s and $2$'s within each period, respectively. Then $(p-q)(k/T) = 1$, which implies that $p - q = k/T = 1$. Hence, $k$ is the period of $(\Gamma(k, n))_{n\ge 1}$.
\end{proof}

\begin{proof}[Proof of Theorem \ref{m4} when $k$ is even]
Let $u, v\in \mathbb{N}$ such that $v = u+2k$. We show that $\Gamma(k, u) = \Gamma(k, v)$. 

\mbox{ }

Case 1: $u$ divides $k$, giving $\Gamma(k, u) = 1$. Write $k = u\ell$ for some $\ell\in \mathbb{N}$. Then $v = u + 2k = u(2\ell+1)$. We have
$\Gamma(k, v) = \Gamma(u\ell, u(2\ell+1)) = \Gamma(\ell, 2\ell+1)$. Since 
$$\ell (\ell-1) + (2\ell+1)\times 0\ =\ \frac{(\ell-1)((2\ell+1)-1)}{2}\,,$$
$\Gamma(\ell, 2\ell+1) = 1$; hence, $\Gamma(k, u) = \Gamma(k, v) = 1$. 

\mbox{ }

Case 2: $k$ divides $u$. Then $k$ divides $v$, and $\Gamma(k, u) = \Gamma(k, v) = 1$.

\mbox{ }

Case 3: $u$ does not divide $k$, and $k$ does not divide $u$. Let $d = \gcd(k, u) = \gcd(k, v)$. 
If $k/d$ is odd, the exact same argument as in the proof of Theorem \ref{m4} when $k$ is odd applies. Suppose that $k/d$ is even.
By Theorem \ref{m1}, $\Gamma(k, u)$ is determined by the parity of $x := \Theta(k, u)$. Similarly, $\Gamma(k, v)$ is determined by the parity of $y:= \Theta(k, v)$. Hence, it suffices to show that $x$ and $y$ have the same parity. By definition,
$$\frac{kx}{d}\ \equiv\  1 \mod \frac{u}{d}\ \Rightarrow\ \frac{kx}{d}-1\ =\ \frac{u}{d}\ell_1\,,\ \ \ \ell_1>0\,,\ \ell_1\ \text{is odd,}$$
$$\frac{ky}{d}\ \equiv\  1 \mod \frac{v}{d}\ \Rightarrow\ \frac{ky}{d}-1\ =\ \frac{u}{d}\ell_2+2\frac{k}{d}\ell_2\,, \ \ \ \ell_2>0\,,\ \ell_2\ \text{is odd.}$$
Note that $\ell_1$, $\ell_2 < k/d$ as $x < u/d$ and $y < v/d\ =\ (u+2k)/d$. Subtracting the two equations above gives
\begin{equation}\label{Eq: Gamma(k,N)_k even_Case 3}
    \frac{k}{d}(x-y)\ =\ \frac{u}{d}(\ell_1-\ell_2)-2\frac{k}{d}\ell_2\,,
\end{equation}

\noindent which implies that $\ell_1-\ell_2$ is divisible by $k/d$. However, $-k/d<\ell_1-\ell_2<k/d$; therefore, $\ell_1-\ell_2=0$. Replacing $\ell_1-\ell_2$ by $0$ in \eqref{Eq: Gamma(k,N)_k even_Case 3}, we obtain
$$x-y\ =\ -2\ell_2\,.$$
As a result, $x \equiv y \mod 2$.

Next, we prove that within the first $2k$ terms of $(\Gamma(k, n))_{n\ge 1}$, the number of $1$'s is two more than the number of $2$'s. Pick $1\le s\le k-1$ and let $r = \gcd(k,s) = \gcd(k, 2k-s)$. We show that $\Gamma(k,s)\neq \Gamma(k, 2k-s)$. Then we are done since $\Gamma(k, k) = \Gamma(k, 2k) = 1$. 

\mbox{ }

Case I: if $s$ divides $k$, then $\Gamma(k, s) = 1$. Write $k = s\ell$ for some $\ell\in \mathbb{N}_{\ge 2}$. By Corollary  \ref{c1}, we have
$\Gamma(k, 2k-s) = \Gamma(\ell, 2\ell-1) = 2$. Hence, $\Gamma(k,s)\neq \Gamma(k, 2k-s)$.

\mbox{ }

Case II: suppose that $s$ does not divide $k$. Observe that $k+1 \le 2k-s \le 2k-1$, so $k$ does not divide $2k-s$. If $k/r$ is odd, the exact same argument as in the proof of Theorem \ref{m4} when $k$ is odd applies. Suppose that $k/r$ is even. 
By Theorem \ref{m1}, $\Gamma(k, s)$ and $\Gamma(k, 2k-s)$ are determined by the parity of $p:=\Theta(k, s)$ and $q:=\Theta(k, 2k-s)$, respectively. By definition,
$$\frac{pk}{r}\ \equiv\ 1\mod \frac{s}{r}\mbox{ and }\frac{qk}{r}\ \equiv\ 1\mod \frac{2k-s}{r}\,.$$
Write 
\begin{align*}
\frac{pk}{r} - 1 &\ =\ \frac{s}{r}\ell_1\,, \quad \ell_1> 0\,, \ell_1\mbox{ odd,}\\
\frac{qk}{r} - 1 &\ =\ \frac{2k-s}{r}\ell_2\,, \quad \ell_2 > 0\,, \ell_2\mbox{ odd}. 
\end{align*}
Note that $\ell_1, \ell_2 < k/r$ because $p < s/r$ and $q < (2k-s)/r$.
Hence, 
\begin{equation}\label{e50}\frac{k}{r}(p-q)\ =\ \frac{s}{r}(\ell_1+\ell_2) - \frac{2k}{r}\ell_2\,,\end{equation}
which implies that $\ell_1+\ell_2$ is divisible by $k/r$. However, $0 < \ell_1 + \ell_2 < 2k/r$; therefore, $\ell_1 + \ell_2 = k/r$. Replacing $\ell_1 + \ell_2$ by $k/r$ in \eqref{e50}, we obtain
$$p-q \ =\ \frac{s}{r} - 2\ell_2\,.$$
As a result, $p\not\equiv q\mod 2$, so $\Gamma(k, s)\neq \Gamma(k, 2k-s)$. 

Finally, we prove that $(\Gamma(k, n))_{n\ge 1}$ has period $2k$. Let $T$ be the period of $(\Gamma(k, n))_{n\ge 1}$. By Lemma \ref{l2}, $T$ divides $2k$. Hence, within the first $2k$ terms, there are $2k/T$ copies of the period. Let $p$ and $q$ be the number of $1$'s and $2$'s within each period, respectively. Then $(p-q)(2k/T) = 2$; equivalently, $(p-q)k/T = 1$. Since $p-q\le 2$ by above, there are two cases: either $(p-q, T) = (1, k)$ or $(p-q, T) = (2, 2k)$. We show that the former cannot happen. Suppose, for a contradiction, that $T = k = 2j$ for some $j\in \mathbb{N}$. It follows that $\Gamma(2j, 2j+1) = \Gamma(k, k+1) = \Gamma(k, 1) = 1$, contradicting Lemma \ref{l3}. This completes our proof. 
\end{proof}

\section{Fibonacci-Type Recurrence}\label{s5}
In this section, we find $\Delta((a_n)_n)$, where $(a_n)_n$ satisfies a certain linear recurrence of order two. This extends \cite[Theorem 1.4]{C}.
\begin{lemma}\label{lemsame}
     Fix $a,b,k\in \mathbb{N} $ with $\gcd(a,b)\ = \ d$. Consider the sequence $(a_n)_n$ where $a_1=a,a_2=b$, and $a_n = ka_{n-1}+a_{n-2}$ for $n\ge 3$. Then $\gcd(a_n,a_{n+1})=d$ for all $n\in \mathbb{N}$.
\end{lemma}

\begin{proof}
    We prove by induction. For $n=1$, $\gcd(a_1,a_2)=\gcd(a,b)=d$. Suppose $\gcd(a_n,a_{n+1})=d$ for some $n\ge 1$.
    We have 
    $$\gcd(a_{n+2}, a_{n+1})\ =\ \gcd(ka_{n+1}+a_n, a_{n+1})\ =\ \gcd(a_{n}, a_{n+1})\ =\ d\,.$$
    This completes our induction step. 
\end{proof}

Due to Lemma \ref{lemsame}, if $(a_n)_n$ satisfies $a_n = ka_{n-1}+a_{n-2}$, then in finding $\Delta((a_n)_n)$, we can assume that $\gcd(a_1, a_2) = 1$ with any loss of generality. 

\begin{theorem}\label{gcd1}
Fix $a,b,k \in \mathbb{N}$ with $\gcd(a,b) = 1$. Let $a_1=a,a_2=b$, and $a_n = ka_{n-1}+a_{n-2}$ for $n\ge3$. Let $(x_n)_n$ be the sequence $1,1,1,2,2,2,1,1,1,\ldots$ and $(y_n)_n$ be the sequence $2,2,2,1,1,1,2,2,2,\ldots$. 
\begin{enumerate}
\item If $k$ is even, $\Delta((a_n)_n)$ is constant.
\item If $k$ is odd, 
\begin{enumerate}
    \item and $a,b$ are both odd, $\Delta((a_n)_{n\ge3}) =  (x_n)_n$ or $(y_n)_n$, or if 
    \item $a$ is odd, and $b$ is even, $\Delta((a_n)_{n\ge2}) = (x_n)_n$ or $(y_n)_n$, or if
    \item $a$ is even, and $b$ is odd, $\Delta((a_n)_n) = (x_n)_n$ or $(y_n)_n$.
\end{enumerate}
\end{enumerate}
\end{theorem}
\begin{proof}
Fix $n\in \mathbb{N}$ and consider three consecutive terms $a_{n},a_{n+1},a_{n+2}$, which are
$a_n, a_{n+1}, ka_{n+1}+a_n$. We proceed by case analysis. 

\mbox{ }

Case 1: $a_{n+1}$ is odd and $a_{n+1}\neq 1$. We show in this case that $\Gamma(a_n, a_{n+1}) = \Gamma(a_{n+1}, a_{n+2})$. 

\begin{enumerate}

\item Case 1.1: $a_n$ divides $a_{n+1}$. Write $a_{n+1} = a_n\ell$ for some odd $\ell\in \mathbb{N}$. Then $a_{n+2} = a_n(k\ell+1)$, which gives $\Gamma(a_{n+1}, a_{n+2}) = \Gamma(\ell, k\ell+1) = 1$ because
$$\ell \times \frac{(\ell-1)k}{2} + (k\ell+1)\times 0\ =\ \frac{(\ell-1)k\ell}{2}\,.$$
Therefore, $\Gamma(a_n, a_{n+1}) = \Gamma(a_{n+1}, a_{n+2})$.

\item Case 1.2: $a_n$ does not divide $a_{n+1}$. 
We claim that $a_{n+1}$ divides neither $a_{n}$ nor $a_{n+2}$. Indeed, if $a_{n+1}|a_{n}$, then $\gcd(a_n,a_{n+1}) = a_{n+1} > 1$, contradicting $\gcd(a_{n},a_{n+1})=\gcd(a,b) = 1$ by Lemma \ref{lemsame}. Similarly, $a_{n+1}$ does not divide $a_{n+2}$.  

Let $c = \Theta(a_n,a_{n+1})$ and $d = \Theta(a_{n+2},a_{n+1})$ to have
\begin{align}\label{1} ca_n &\ \equiv \ 1\mod a_{n+1}\,,\\
\label{2}d(ka_{n+1}+a_n)&\ \equiv\ 1 \mod a_{n+1}\ \Longrightarrow\ a_nd\ \equiv\ 1\mod a_{n+1}\,.\end{align}
Therefore, $c = d$, which, by Theorem \ref{m1}, gives $\Gamma(a_n,a_{n+1}) = \Gamma(a_{n+1},a_{n+2})$.
\end{enumerate}

\mbox{ }

Case 2: $a_{n+1}$ is odd and $a_{n+1} = 1$. Clearly, $\Gamma(a_n, a_{n+1}) = \Gamma(a_{n+1}, a_{n+2}) = 1$. 

From Case 1 and Case 2, we conclude
\begin{equation}\label{e101}\Gamma(a_{n},a_{n+1})\ =\ \Gamma(a_{n+1},a_{n+2})\,, \quad \mbox{ if }a_{n+1}\mbox{ is odd.}\end{equation}

\mbox{ }

Case 3: $a_{n+1}$ is even and $a_{n}\neq 1$. 

\begin{enumerate}

\item Case 3.1: $a_{n+1}$ divides $a_n$. Then $\gcd(a_n, a_{n+1}) = 1$ implies that $a_{n+1} = 1$, contradicting that $a_{n+1}$ is even.

\item Case 3.2: $a_{n+1}$ does not divide $a_n$. As in above cases, $a_n$ does not divide $a_{n+1}$, and $a_{n+1}$ does not divide $a_{n+2}$ because otherwise, we violate the condition $\gcd(a_n, a_{n+1}) = \gcd(a_{n+1}, a_{n+2}) = 1$. Let $p=\Theta(a_{n+1},a_n)$ and $q = \Theta(a_{n+1},a_{n+2})$ to have
\begin{align*}
    pa_{n+1} &\ \equiv\ 1 \mod a_n\,,\\
qa_{n+1}&\ \equiv\ 1 \mod ka_{n+1}+a_n\,.
\end{align*}
There exist positive integers $k_1, k_2$ such that
\begin{align}\label{5}pa_{n+1} &\ = \ 1+k_1a_n\,, \\
\label{6}qa_{n+1} &\ = \ 1+k_2(ka_{n+1}+a_n)\,. \end{align}
By definition,  $0< p< a_{n}$ and $0 < q < ka_{n+1}+a_n$. Hence, \eqref{5} and \eqref{6} imply that 
$0 < k_1, k_2 <a_{n+1}$.
Subtracting \eqref{6} from \eqref{5} side by side, we obtain 
\begin{equation}\label{7}a_{n+1}(p-q+kk_2) \ = \ a_n(k_1-k_2)\,. \end{equation}
Since $\gcd(a_n,a_{n+1})=1$, it follows from \eqref{7} that $a_{n+1}$ divides $k_1-k_2$. However, that $0 < k_1, k_2 < a_{n+1}$ implies that $-a_{n+1} < k_1-k_2 < a_{n+1}$. As a result, $k_1$ must be equal to $k_2$. Replacing $k_1$ by $k_2$ in \eqref{7} gives
\begin{equation}\label{e100}p-q + kk_2 \ =\ 0\,.\end{equation}
From \eqref{6} and the fact that $a_{n+1}$ is even, we know that $k_2$ is odd. Therefore, from \eqref{e100}, we conclude that
\begin{itemize}
\item if $k$ is even, $p$ and $q$ have the same parity, and 
\item if $k$ is odd, $p$ and $q$ have different parities. 
\end{itemize}
\end{enumerate}

Therefore, when $a_{n+1}$ is even and $a_n \neq 1$,
\begin{itemize}
\item if $k$ is even, then  $\Gamma(a_n,a_{n+1})=\Gamma(a_{n+1},a_{n+2})$, while
\item if $k$ is odd, then $\Gamma(a_n,a_{n+1})\neq \Gamma(a_{n+1},a_{n+2})$.
\end{itemize}

\mbox{ }

Case 4: $a_{n+1}$ is even and $a_n = 1$. Write $a_{n+1}=2r$ for some $r\in \mathbb{N}$ to have $a_{n+2} = 2kr 
 + 1$. Clearly, $\Gamma(a_n,a_{n+1})=1$. We find $\Gamma(a_{n+1}, a_{n+2})$. Note that $\gcd(2r,2kr+1)=1$ and $\Theta(2r,2kr+1) = 2kr+1-k$, which is odd if and only if $k$ is even. Therefore, if $k$ is even, $\Gamma(a_n,a_{n+1})=\Gamma(a_{n+1},a_{n+2})=1$, and if $k$ is odd, then $\Gamma(a_n,a_{n+1})=1\ne 2 = \Gamma(a_{n+1},a_{n+2})$.\\

From Case 3 and Case 4, we conclude that 
\begin{align}
\label{e102}\Gamma(a_n,a_{n+1})&\  =\  \Gamma(a_{n+1},a_{n+2})\,, \quad \mbox{ if } a_{n+1} \mbox{ is even and } k \mbox{ is even.}\\
\label{e103}\Gamma(a_n,a_{n+1}) &\ \neq\  \Gamma(a_{n+1},a_{n+2})\,, \quad \mbox{ if } a_{n+1} \mbox{ is even and } k \mbox{ is odd.}
\end{align}

From \eqref{e101}, \eqref{e102}, and \eqref{e103}, we have
 \begin{itemize}
\item If $k$ is even, then $\Gamma(a_n,a_{n+1})=\Gamma(a_{n+1},a_{n+2})$ for all $n\in \mathbb{N}$.
\item If $k$ is odd, then $\Gamma(a_n,a_{n+1})=\Gamma(a_{n+1},a_{n+2})$ if and only if  $a_{n+1}$ is odd.
 \end{itemize}
Hence, we have proved Item (1) of Theorem \ref{gcd1}. For Item (2), assume that $k$ is odd. 

If $a$ and $b$ are both odd, the parity of terms in $(a_n)_n$ is 
$$\mbox{odd } , \mbox{odd } , \mbox{even } , \mbox{odd } , \mbox{odd } , \mbox{even } , \mbox{odd } , \mbox{odd } , \mbox{even } , \ldots\,.$$
It is easy to verify that \eqref{e103} gives us Item (2) Part (a). 

If $a$ is odd and $b$ is even, the parity of terms in $(a_n)_n$ is 
$$\mbox{odd } , \mbox{even } , \mbox{odd } , \mbox{odd } , \mbox{even } , \mbox{odd } , \mbox{odd } , \mbox{even } , \mbox{odd } , \ldots\,.$$
In this case, \eqref{e103} gives us Item (2) Part (b).

Similarly, we get Item (3) Part (c) when $a$ is even and $b$ is odd. 
\end{proof}

As discussed before, we can drop the condition $\gcd(a, b) = 1$ in Theorem \ref{gcd1}.

\begin{corollary}
Fix $a,b,k \in \mathbb{N}$. Let $a_1=a,a_2=b$, and $a_n = ka_{n-1}+a_{n-2}$ for $n\ge3$. Let $(x_n)_n$ be the sequence $1,1,1,2,2,2,1,1,1,\ldots$ and $(y_n)_n$ be the sequence $2,2,2,1,1,1,2,2,2,\ldots$. 
\begin{enumerate}
\item If $k$ is even, $\Delta((a_n)_n)$ is constant.
\item If $k$ is odd, 
\begin{enumerate}
    \item and $a,b$ are both odd, $\Delta((a_n)_{n\ge3}) =  (x_n)_n$ or $(y_n)_n$. 
    \item $a$ is odd, and $b$ is even, $\Delta((a_n)_{n\ge2}) = (x_n)_n$ or $(y_n)_n$.
    \item $a$ is even, and $b$ is odd, $\Delta((a_n)_n) = (x_n)_n$ or $(y_n)_n$.
\end{enumerate}
\end{enumerate}
\end{corollary}

\begin{proof}
The corollary follows immediately from how we define $\Delta((a_n)_n)$. 
\end{proof}
\section{Problems for Future Investigation}\label{op}

It would be interesting to see $\Delta((a_n)_n)$ when $(a_n)_{n=1}^\infty$ satisfies more general recurrences to extend results in Sect. \ref{s5}. For other problems in this topic, interested readers may refer to \cite{CMT}.

\section*{Acknowledgement.} The authors thank Garrett Tresch for pointing out and fixing an error in the proof of Theorem \ref{m4} for even $k$.

%
%

\end{document}